\documentclass[10 pt]{amsart}

\usepackage{amssymb,amsmath,amsfonts,mathrsfs}

\newtheorem{theorem}{Theorem}[section]
\newtheorem{thm}{Theorem}[section]
\newtheorem{lemma}[theorem]{Lemma}
\newtheorem{lem}[theorem]{Lemma}
\newtheorem{cor}[theorem]{Corollary}
\newtheorem{prop}[theorem]{Proposition}

\theoremstyle{definition}

\theoremstyle{remark}
\newtheorem{remark}[theorem]{Remark}
\newtheorem{rem}[theorem]{Remark}
\numberwithin{equation}{section}

\newcommand{\xcomp}{C_{c}(X)}
\newcommand{\lw}{{\ell^{2}_{\mu}}(X)}

\newcommand{\lwi}{{\ell^{\infty}}(X)}

\newcommand{\dom}{\operatorname{D}}

\newcommand{\ZZ}{\mathbb{Z}}

\title[Discrete Schr\"odinger operators]{On two properties of positively perturbed discrete Schr\"odinger operators}
\author{Ognjen Milatovic}
\address{Department of Mathematics
and Statistics \\ University of North Florida \\ Jacksonville, FL
32224 \\ USA.}
\email{omilatov@unf.edu}

\begin{document}
\maketitle

\begin{abstract} We show that if we start from a symmetric lower semi-bounded Schr\"odinger operator $\mathcal{H}$ on finitely supported functions on a discrete weighted graph (satisfying certain conditions), apply the Friedrichs construction to get a self-adjoint extension $H$, and then perturb $H$ by a non-negative function $W$, then the resulting form-sum $H\widetilde{+}W$ coincides with the Friedrichs extension of $\mathcal{H}+W$. Additionally, we consider a non-negative perturbation of an essentially self-adjoint lower semi-bounded Schr\"odinger operator $H$ on a discrete weighted graph. We show that, under certain conditions on the graph and the perturbation, the essential self-adjointness of $H$ remains stable under the given perturbation.
\end{abstract}

\section{Introduction}\label{S:main} About fifteen years ago, in their seminal papers~\cite{Keller-Lenz-10, Keller-Lenz-09,W-08}, the authors established a convenient framework for analysis on (discrete) weighted graphs, thereby propelling the study of existence and uniqueness of  self-adjoint extensions of the (discrete) Laplacian and, more generally, Schr\"odinger operators. Over the ensuing years, several researchers have worked on this (and related) question(s), resulting in quite a few articles; see, for instance~\cite{vtt-11-2,vtt-11-3,Gol,HKMW,Jor-08, Keller-Lenz-10,Keller-Lenz-09,KN-22,Milatovic-Truc,MS-18,Torki-10}. With regard to the self-adjointness of (primarily lower semi-bounded) magnetic Schr\"odinger operators (acting on vector bundles) over infinite (not necessarily locally finite) graphs, the most general results thus far are contained in~\cite{MS-18} and, for (lower semi-bounded) Schr\"odinger operators acting on scalar functions on discrete (as well as metric) graphs, in the paper~\cite{KN-22}. Interested readers can find additional references on the self-adjointness of discrete Laplace/Schr\"odinger operators on graphs in the papers~\cite{Keller-15,KN-22,KN-23,MS-18} and the monographs~\cite{KLW-book,KN-22-book}.

In parallel with the aforementioned developments, motivated by~\cite{Masamune-09}, some authors have studied Laplace-type operators acting on 1-forms; see, for instance,~\cite{AT-15,BGJ-15}.  Subsequently, the essential self-adjointness of the Laplacian was studied on a 2-simplicial complex in~\cite{CY-17}, and on a triangulation (in presence of a magnetic potential) in~\cite{AACTH-23}. Furthermore, in the recent paper~\cite{IKMW-24}, the authors investigated the self-adjointness (and other properties) of the weighted Laplacian on birth-death chains. Interesting connections between essential self-adjointness and other properties (such as $L^2$-Liouville property) of the Laplacian on graphs (and manifolds) were explored in~\cite{HMW-21}. Additionally, far-reaching studies of (quadratic) form extensions (with applications to graphs) were conducted by the authors of~\cite{LSW-17,LSW-16}. Lastly, the authors of~\cite{ABTH-19} and~\cite{ABTH-20}, studied various properties of non-self-adjoint operators.

In the present paper, we address two phenomena exhibited by positive perturbations of lower semi-bounded Schr\"odinger operators acting on weighted  graphs.  The origin of the first phenomenon lies a problem discussed in section VI.5.2 of the book~\cite{Kato80}. More specifically, let $t_1$ and $t_2$ be two closed and lower semi-bounded forms in a Hilbert space $\mathscr{H}$, with the associated (self-adjoint and lower semi-bounded) operators $T_1$ and $T_2$. To the form $t_1+t_2$, which is closed and lower semi-bounded, we can associate a (self-adjoint and lower semi-bounded) operator $T_1\widetilde{+}T_2$, called the form-sum of $T_1$ and $T_2$. On the other hand, if the operator sum $T_1+T_2$ happens to be densely defined, we can certainly get the Friedrichs extension $T_{F}$ of $T_1+T_2$.  As demonstrated in example VI.2.19 of~\cite{Kato80}, the form-sum $T_1\widetilde{+}T_2$ generally differs from $T_{F}$. Motivated by this  discussion, we consider the question of coincidence of the following two operators: (i)  the form-sum $L_{V_1}\widetilde{+}M_{V_2}$, where  $M_{V_2}$ is the  maximal multiplication operator by a \emph{non-negative} function $V_2$ (here, $M_{V_2}$ acts in the usual $\lw$-space, where $X$ is a discrete (weighted) graph with vertex weight $\mu$), and $L_{V_1}$ is the Friedrichs extension of a discrete lower semi-bounded (symmetric) Schr\"odinger operator $\mathcal{L}_{V_1}$ (with real-valued potential $V_1$) defined initially on $\xcomp$, the space of finitely supported functions on $X$; (ii) the Friedrichs extension ${L}_{V_1+V_2}$ of the Schr\"odinger operator $\mathcal{L}_{V_1+V_2}|_{\xcomp}$ with potential $V_1+V_2$. In theorem~\ref{T:main-1} we show that the operators (i) and (ii) coincide in the context of weighted graphs satisfying the so-called \emph{finiteness condition}: for all $x\in X$ the mapping $y\mapsto b(x,y)/\mu(y)$ belongs to $\lw$, where $b(x,y)$ stands for the edge weight; see sections~\ref{SS:setting} and~\ref{SS:expressions} for more details regarding the notations. For Schr\"odinger operators in $L^2({\mathbb{R}}^n)$, this phenomenon was first discovered by the author of~\cite{cyc-81}. Later, the author of~\cite{M-04} observed that this property carries to Schr\"odinger operators in $L^2(X)$, where $X$ is a non-compact Riemannian manifold.

We now move on to describe the second phenomenon. If we start with a lower semi-bounded and essentially self-adjoint Schr\"odinger operator $\mathcal{L}_{V_1}|_{\xcomp}$ and perturb it by a \emph{non-negative} potential $V_2$ which is $\mathcal{L}_{V_1}$-bounded (that is, dominated by $\mathcal{L}_{V_1}$ in the operator sense; see inequality~(\ref{E:a-1-2}) for a precise description) with bound $\alpha\geq 0$, then  the operator $\mathcal{L}_{V_1+V_2}|_{\xcomp}$ will be essentially self-adjoint. As $\alpha$ is not assumed to satisfy $0\leq\alpha<1$ (or $\alpha=1$), we cannot use the celebrated Kato--Rellich theorem (or W\"ust's theorem) to infer the essential self-adjointness of $\mathcal{L}_{V_1+V_2}|_{\xcomp}$. In theorem~\ref{T:main-2} we prove that this kind of stability under a positive perturbation holds in the setting of weighted graphs satisfying the finiteness condition (in the sense of the preceding paragraph). For Schr\"odinger operators acting in $L^2({\mathbb{R}}^n)$, this phenomenon was first discovered by the author of~\cite{Kap-85}. Later it was noted in~\cite{M-05} that the same phenomenon occurs in Schr\"odinger operators acting in $L^2(X)$, where $X$ is a non-compact Riemannian manifold.

With regard to theorem~\ref{T:main-2}, in the context of birth-death chains (that is, weighted graphs $(X,b,\mu)$ with vertex set $X=\{0,1,2,\dots\}$ and edge weights $b$ such that $b(x,y)>0$ if and only if $|x-y|=1$), a stronger result was established recently in proposition 4.7 of~\cite{IKMW-24}: If the weighted Laplacian $\mathcal{L}$ (as in~(\ref{E:magnetic-lap}) below with $V=0$) is essentially self-adjoint on $\xcomp$ and if $W\in C(X)$ satisfies $W(x)\geq K$ for all $x\in X$, where $K\in\mathbb{R}$, then $\mathcal{L}+W$ is essentially self-adjoint on $\xcomp$. Finally, in remark~\ref{R:rem-gk} below, we describe the relationship between our theorem~\ref{T:main-2} and the essential self-adjointness result of theorem 2.16(b) in~\cite{GKS-15} regarding lower semi-bounded Schr\"odinger operators on graphs $(X, b, \mu)$ satisfying metric completeness condition.

An important tool for demonstrating the two phenomena described above is the \emph{positivity-preserving property} for the resolvent of a (lower semi-bounded and self-adjoint) Schr\"odinger operator $H^{0}_{V_1}$ associated to the closure of a (closable) lower semi-bounded Schr\"odinger form with a potential $V_1$. In the setting of general weighted graphs (not necessarily satisfying the finiteness condition), this property (which we recalled for convenience in proposition~\ref{P:P-1} below) was established by the authors of~\cite{GKS-15} as a consequence of the (discrete) Feynman--Kac--It\^o formula, also proved in~\cite{GKS-15}. Later, the author of~\cite{MS-18} gave a purely analytic proof of the mentioned positivity-preserving property.   In the course of proving our theorem~\ref{T:main-1}, we establish the \emph{positive form core property} for the Friedrichs extension $L_{V_1}$ of a lower-semi bounded Schr\"odinger operator $\mathcal{L}_{V_1}|_{\xcomp}$. The latter property says that every \emph{non-negative} function $u$ belonging to the form-domain of $L_{V_1}$ can be approximated, in the form norm of $L_{V_1}$, by non-negative elements of $\xcomp$; see section~\ref{SS:pfc} for a precise description. The mentioned positive form core property (see proposition~\ref{P:P-2} below) may be of independent interest.  Lastly, our analysis is facilitated by Kato's inequality for discrete weighted Laplacian (as formulated in~\cite{MS-18}) and a discrete version of the Green's formula (as formulated in~\cite{HK-11}).

The article consists of five sections, including the introduction. In section~\ref{S:main}, after describing the setting, notations, and operators, we state the main results--theorem~\ref{T:main-1} and theorem~\ref{T:main-2}. In section~\ref{S:prelim} we collect some definitions and preliminary facts and prove the positive form core property. The last two sections are devoted to the proofs of theorem~\ref{T:main-1} and theorem~\ref{T:main-2}.

\section{Main Results}\label{S:main}  In this section, we describe weighted graphs, the corresponding function spaces, and discrete Schr\"odinger operators. At the end of the section we state the main theorems.
\subsection{Weighted Graphs and Basic Function Spaces}\label{SS:setting}
By a \emph{symmetric weighted graph} $(X,b)$ we mean a countable \emph{vertex set} $X\neq \emptyset$ together with an \emph{edge weight function} $b\colon X\times X\to[0,\infty)$ satisfying the following conditions:
\begin{enumerate}
  \item [(b1)] $b (x, y) = b (y, x)$, for all $x,\,y\in X$;
  \item [(b2)] $b(x,x)=0$, for all $x\in X$;
  \item [(b3)] $\sum_{z\in X} b(x,z)<\infty$.
\end{enumerate}
We say that the graph $(X,b)$ is \emph{locally finite} if for all $x\in X$ we have $\displaystyle\sharp\,\{y\in X\colon b(x,y)>0\}<\infty$, where $\sharp\, G$ is understood as the number of elements in the set $G$.

We say that the two vertices $x,\, y\in X$  are \emph{connected by an edge} and write $x\sim y$ if $b(x, y) > 0$.  By a \emph{finite path} we mean a finite sequence of vertices $x_0,\,x_1,\,\dots,x_n$ such that $x_{j}\sim x_{j+1}$ for all $0\leq j\leq n-1$. We call the graph $(X,b)$ \emph{connected} if every two vertices $x,\,y\in X$ can be joined by a finite path.

We endow the set $X$ with the discrete topology. The symbol $C(X)$ stands for the set of functions $f\colon X\to\mathbb{C}$, that is, continuous (complex-valued) functions with respect to the discrete topology. The symbol $C_{c}(X)$ indicates the set of finitely supported elements of $C(X)$.

By a \emph{weight} we mean a function $\mu\colon X\to (0,\infty)$. Using the formula
\[
\mu(S):=\displaystyle\sum_{x\in S}\mu(x), \qquad S\subseteq X,
\]
the function $\mu$ gives rise to a Radon measure of full support. In this article, we will use the two notions interchangeably.

Given a weight $\mu$, we define the anti-duality pairing $(\cdot,\cdot)_{a}\colon C(X)\times C_{c}(X)\to \mathbb{C}$ by
\begin{equation}\label{E:a-d}
(u,w)_{a}:=\displaystyle\sum_{x\in X}\mu(x)u(x)\overline{w(x)}.
\end{equation}
This pairing gives rise to an anti-linear isomorphism $w\mapsto (\cdot,w)_{a}$ between $C_{c}(X)$ and $C(X)'$. Here, $C(X)'$ is the continuous dual of $C(X)$, where the latter space is equipped with the topology of pointwise convergence.

The symbol $\lw$ stands for the space of functions $f\in C(X)$ such that
\begin{equation}\label{E:l-p-def}
\displaystyle\sum_{x\in X}\mu(x)|f(x)|^2<\infty,
\end{equation}
where $|\cdot|$ denotes the modulus of a complex number.

The space $\lw$ is a Hilbert space with the inner product
\begin{equation}\label{E:inner-w}
(f,g):=\sum_{x\in X}\mu(x)f(x)\overline{g(x)}.
\end{equation}
The inner product $(\cdot,\cdot)$ in $\lw$ induces the norm $\|\cdot\|$. Comparing with $(\cdot,\cdot)_{a}$, we see that if $f\in\lw$ and $g\in\xcomp$, then $(f,g)=(f,g)_{a}$.

\subsection{Finiteness Condition (FC)}\label{SS:FC}
Following the terminology of definition 3.16 in~\cite{MS-18}, we say that the triplet $(X,b,\mu)$ satisfies the \emph{finiteness condition}, abbreviated as (FC), if for all $x\in X$ the mapping $y\mapsto b(x,y)/\mu(y)$ belongs to $\lw$.

\subsection{Discrete Schr\"odinger Operators}\label{SS:expressions}
For a weighted graph $(X,b)$, a weight $\mu\colon X\to (0,\infty)$, and a real-valued function $V\in C(X)$ (called \emph{potential}), we define the \emph{formal discrete Schr\"odinger operator} $\mathcal{L}_{V}\colon \mathcal{F}\to C(X)$ as follows: the domain of $\mathcal{L}_{V}$ is described as
\begin{equation}\label{E:def-f}
\mathcal{F}:=\{f\in C(X)\colon \displaystyle\sum_{y\in X}b(x,y)|f(y)|<\infty\textrm{ for all } x\in X\},
\end{equation}
and the action of $\mathcal{L}_{V}$ is described as
\begin{equation}\label{E:magnetic-lap}
    \mathcal{L}_{V}f(x):=\frac{1}{\mu(x)}\sum_{y\in X}b(x,y)(f(x)-f(y))+V(x)f(x),\quad x\in X.
\end{equation}
(To make the notation shorter, we suppressed $\mu$ in the symbol $\mathcal{L}_{V}$.) The absolute convergence of the sum is guaranteed by the description of the domain $\mathcal{F}$. It turns out that $\mathcal{F}=C(X)$ if and only if $(X,b)$ is locally finite.

\subsection{Statements of the results} Throughout this subsection we assume that $(X,b,\mu)$ satisfies the finiteness condition (FC).
By lemma 3.15 in~\cite{MS-18}, the property (FC) is equivalent to the inclusion $\mathcal{L}_{V}[\xcomp]\subseteq\lw$. This means that, in this context, we can consider $\mathcal{L}_{V}|_{\xcomp}$ as an operator in $\lw$. Moreover, in view of Green's formula (see lemma~\ref{L:L-1} below), for any real-valued function $V\in C(X)$, the operator ${L}_{V}|_{\xcomp}$ is symmetric (as an operator in $\lw$).

In our first theorem we assume that $V=V_1+V_2$ where $V_j\in C(X)$ are real-valued, $V_2(x)\geq 0$ for all $x\in X$, and $\mathcal{L}_{V_1}|_{\xcomp}$ is lower semi-bounded. The latter means that there exists $C\in\mathbb{R}$ such that for all $u\in \xcomp$ we have
\begin{equation}\label{E:l-v-1}
(\mathcal{L}_{V_1}u,u)\geq C\|u\|^2.
\end{equation}
The quadratic form
\begin{equation*}
q_{V_2}(u):=\sum_{x\in X}\mu(x) V_2(x)|u(x)|^2,
\end{equation*}
with the domain
\begin{equation*}
\textrm{D}(q_{V_2}):=\{u\in \lw\colon \sum_{x\in X}\mu(x) V_2(x)|u(x)|^2<\infty\},
\end{equation*}
is non-negative and closed. The associated self-adjoint operator $M_{V_2}$ is just the \emph{maximal multiplication operator}: $M_{V_2}u:=V_2u$, where
\begin{equation}\label{d-v-2}
u\in\dom(M_{V_2}):=\{u\in \lw\colon V_2u\in \lw\}.
\end{equation}

As $\mathcal{L}_{V_1}|_{\xcomp}$ is a symmetric and lower semi-bounded operator in $\lw$, the quadratic form $u\mapsto (\mathcal{L}_{V_1}u,u)$, with domain $\xcomp$, is closable (by an abstract fact) and its closure, denoted by $h_{1}$, is a closed and lower semi-bounded form. The associated self-adjoint operator will be denoted by $L_{V_1}$. In the literature, the latter operator is known as the \emph{Friedrichs extension} of $\mathcal{L}_{V_1}|_{\xcomp}$.

As $h_{1}$ and $q_{V_2}$ are closed and lower semi-bounded forms, so is the sum $h_{1}+q_{V_2}$. The associated self-adjoint operator will be denoted by $L_{V_1}\widetilde{+}M_{V_2}$. The latter operator is known in the literature under the name \emph{the form-sum} of $L_{V_1}$ and $M_{V_2}$.

On the other hand, in view of~(\ref{E:l-v-1}) and $V_2\geq 0$, in the same way as in the discussion of $\mathcal{L}_{V_1}|_{\xcomp}$, we obtain the Friedrichs extension $L_{V}$ of $\mathcal{L}_{V}|_{\xcomp}$. This brings us to the first theorem:

\begin{thm}\label{T:main-1} Let $(X, b, \mu)$ be a weighted and connected graph satisfying the property (FC) as in section~\ref{SS:FC}. Assume that $V=V_1+V_2$, where $V_j\in C(X)$ are real-valued, $V_2(x)\geq 0$ for all $x\in X$, and $V_1$ satisfies the condition~(\ref{E:l-v-1}). Let $L_{V_1}$ and $L_{V}$ be the Friedrichs extensions of $\mathcal{L}_{V_1}|_{\xcomp}$ and $\mathcal{L}_{V}|_{\xcomp}$ respectively, and let $M_{V_2}$ be the maximal multiplication operator in $\lw$ corresponding to $V_2$. Then, the form-sum $L_{V_1}\widetilde{+}M_{V_2}$ coincides with $L_{V}$.
\end{thm}

The next theorem is concerned with the stability of self-adjointness under positive perturbations.

\begin{thm}\label{T:main-2} Let $(X, b, \mu)$ be a weighted and connected graph satisfying the property (FC) as in section~\ref{SS:FC}. Assume that $V=V_1+V_2$, where $V_j\in C(X)$ are real-valued, $V_2(x)\geq 0$ for all $x\in X$, and $V_1$ satisfies the condition~(\ref{E:l-v-1}). Assume that $\mathcal{L}_{V_1}|_{\xcomp}$ is essentially self-adjoint. Additionally, assume that there exist numbers $a_j\geq 0$, $j=1,2$, such that
\begin{equation}\label{E:a-1-2}
\|V_2u\|\leq a_1\|\mathcal{L}_{V_1}u\|+a_2\|u\|,
\end{equation}
for all $u\in\xcomp$. Then, $\mathcal{L}_{V}|_{\xcomp}$ is essentially self-adjoint.
\end{thm}
\begin{rem}\label{R:rem-gk} In the context of locally finite weighted graphs $(X, b, \mu)$ that are metrically complete with respect to an intrinsic metric, theorem 2.16(b) from~\cite{GKS-15} guarantees the essential self-adjointness of $\mathcal{L}_{V_1}$ on $\xcomp$ because $\mathcal{L}_{V_1}$ is lower-semibounded in view of the condition~(\ref{E:l-v-1}). Likewise, in this setting, the quoted theorem from~\cite{GKS-15} tells us that  $\mathcal{L}_{V}|_{\xcomp}$ is essentially self-adjoint (because $V=V_1+V_2$ and $V_2\geq 0$). Note that in theorem~\ref{T:main-2} we do not impose any metric completeness conditions on  $(X, b, \mu)$.

Looking at the hypotheses of theorem~\ref{T:main-2}, we stress that $a_1$ is not assumed to satisfy $0\leq a_1<1$ (or $a_1=1$). Thus, we cannot use Kato--Rellich theorem (or W\"ust's theorem) to quickly infer the essential self-adjointness of $\mathcal{L}_{V}|_{\xcomp}$.
\end{rem}

\section{Preliminaries}\label{S:prelim} In this section we collect some definitions and preliminary facts and prove the positive form core property.

\subsection{Green's Formula} The following useful formula was proven in lemma 4.7 of~\cite{HK-11} (see also lemma 2.1 of~\cite{MS-18} for an extension to the magnetic Schr\"odinger operators):
\begin{lemma} \label{L:L-1}  Let $(X, b, \mu)$ be a weighted graph (not necessarily satisfying the property (FC)). Let $W\in C(X)$ be a real-valued function. Let $\mathcal{F}$ be as in~(\ref{E:def-f}) and let $(\cdot,\cdot)_{a}$ be as in~(\ref{E:a-d}). Then, for all $f\in\mathcal{F}$ and $u\in\xcomp$, we have
\begin{align}\nonumber
&(\mathcal{L}_{W}f,u)_{a}=\sum_{x\in X}\mu(x)\mathcal{L}_{W}f(x)\overline{u(x)}=\sum_{x\in X}\mu(x)f(x)\overline{\mathcal{L}_{W}u(x)}\nonumber\\
&=\frac{1}{2}\displaystyle\sum_{x,y\in X}b(x,y)(f(x)-f(y))\overline{(u(x)-u(y))}+\displaystyle\sum_{x\in X}\mu(x)W(x)f(x)\overline{u(x)},\nonumber
\end{align}
with the sums converging absolutely.
\end{lemma}
\begin{remark}\label{R:R-green} If the graph $(X, b, \mu)$ satisfies the property (FC), then $\lw\subseteq\mathcal{F}$; see lemma 2.15 in~\cite{MS-18}.
\end{remark}

Before going further, we should point out that we are not assuming that $(X, b, \mu)$ satisfies the property (FC) until we formulate corollary~\ref{C:C-1}.

\subsection{Quadratic Forms and Self-adjoint Operators}\label{SS:q-forms}
Here we recall some terminology from section 2.5 in~\cite{MS-18}. For a real-valued $W\in C(X)$, the \emph{space of functions of finite energy} $\mathcal{D}_{W}$ is defined as
\begin{equation*}
\mathcal{D}_{W}:=\{f\in C(X)\colon \displaystyle\sum_{x,y\in X}b(x,y)|f(x)-f(y)|^2+\displaystyle\sum _{x\in X}\mu(x)|W(x)||f(x)|^2<\infty\}.
\end{equation*}
\begin{remark} It turns out (see lemma 2.4 in~\cite{MS-18}) that the inclusion $\mathcal{D}_{W}\subseteq\mathcal{F}$ holds.
\end{remark}
With this notation in place, we define the sesquilinear form $Q_{W}\colon \mathcal{D}_{W}\times \mathcal{D}_{W}\to \mathbb{C}$ by the formula
\begin{equation*}
Q_{W}(f,g):=\frac{1}{2}\displaystyle\sum_{x,y\in X}b(x,y)(f(x)-f(y))\overline{(g(x)-g(y))}+\displaystyle\sum_{x\in X}\mu(x)W(x)f(x)\overline{g(x)}.
\end{equation*}
We use the symbol $Q^{c}_{W}$ to indicate the restriction of $Q_{W}$ to $\xcomp$. To simplify the notations, the symbols $Q_{W}(\cdot)$ and $Q^{c}_{W}(\cdot)$ stand for the quadratic forms corresponding to $Q_{W}(\cdot,\cdot)$ and $Q^{c}_{W}(\cdot,\cdot)$.

If the form $Q^{c}_{W}$ is lower semi-bounded and closable in the space $\lw$, the corresponding closure of $Q^{c}_{W}$ will be denoted by $Q^{0}_{W}$, and the associated (lower semi-bounded) operator will be denoted by $H^{0}_{W}$. The infimum of the spectrum of $H^{0}_{W}$ will be indicated by $\lambda_{0}(H^{0}_{W})$.
\subsection{Positive Form Core}\label{SS:pfc}
Let $h$ be a lower semi-bounded and closed quadratic form in $\lw$ and let $\lambda_0(h)$ be the largest possible $C\in\mathbb{R}$ such that $h(u)\geq C$ for all $u\in\dom(h)$. For $\beta>0$, the notation
\begin{equation}\label{E:n-alpha}
\|\cdot\|^2_{h}:=h(\cdot)+(\beta-\lambda_{0}(h))\|\cdot\|^2,
\end{equation}
where $\|\cdot\|$ is the norm in $\lw$, indicates the form norm corresponding to $h$. While the norm~(\ref{E:n-alpha}) depends on $\beta$, it turns out that the norms corresponding to any pair $\beta_1>0$, $\beta_2>0$ are equivalent.

Let $T$ be a lower-semibounded and symmetric operator in $\lw$ with domain $\dom (T)=\xcomp$. Let $t$ be the closure of the (closable) form $u\mapsto (Tu,u)$, $u\in\xcomp$, and let $T_{F}$ be the (lower semi-bounded) self-adjoint operator corresponding to $t$. (The operator $T_{F}$ is the Friedrichs extension of $T$.)

Following the terminology of~\cite{cyc-81}, we say that $T_{F}$ \emph{has a positive form core} if for all $u\in [\dom(t)]^{+}$ there exists a sequence $u_k\in [\xcomp]^{+}$ such that $\|u_k-u\|_{t}\to 0$ as $k\to\infty$. (Here, $\dom (t)$ stands for the domain of the form $t$, and $G^{+}$ denotes the set of non-negative elements of the set $G\subseteq C(X)$ and $\|\cdot\|_{t}$ is the form norm corresponding to $t$.)

\subsection{Preliminary Lemmas}

The following proposition was proved in~\cite{MS-18}:

\begin{prop}\label{P:P-1} Let $(X, b, \mu)$ be a weighted graph (not necessarily satisfying the property (FC)). Assume that $Q^{c}_{W}$ is lower semi-bounded. Then, the following properties hold:
\begin{enumerate}
  \item [(i)] The form $Q^{c}_{W}$ is closable in $\lw$.
  \item [(ii)] The closure $Q^{0}_{W}$ of $Q^{c}_{W}$ satisfies the first Beurling--Deny condition, that is, if $u\in \dom(Q^{0}_{W})$, then $|u|\in \dom(Q^{0}_{W})$ and $Q^{0}_{W}(|u|)\leq Q^{0}_{W}(u)$.
  \item [(iii)] If $H^{0}_{W}$ is the self-adjoint operator associated to $Q^{0}_{W}$, then $(H^{0}_{W}+\alpha)^{-1}$ is positivity preserving for all $\alpha>-\lambda_0(H^{0}_{W})$, that is, the inequality $0\leq f\in\lw$ implies $(H^{0}_{W}+\alpha)^{-1}f\geq 0$.
\end{enumerate}
\end{prop}
\begin{rem} For the proof of proposition~\ref{P:P-1}, see~\cite{MS-18}; more specifically, the arguments pertaining to theorem 3.8 and lemma 3.9. We should point out that the author of~\cite{MS-18} considers much more general operators (acting on sections of Hermitian vector bundles over graphs). For a probabilistic proof, based on the Feynman--Kac--It\^o formula on graphs, see the paper~\cite{GKS-15}.
\end{rem}

In this article we will use the following special case of proposition~\ref{P:P-1}:

\begin{cor}\label{C:C-1} Let $(X, b, \mu)$ be a weighted graph satisfying the property (FC). Assume that $V_1\in C(X)$ is a real-valued function satisfying the condition~(\ref{E:l-v-1}), with $\mathcal{L}_{V_1}$ as in~(\ref{E:magnetic-lap}). Let $L_{V_1}$ be the Friedrichs extension of $\mathcal{L}_{V_1}|_{\xcomp}$. Then, the following properties hold:
\begin{enumerate}
  \item [(i)] The (closed, semi-bounded) form $h_{1}$ corresponding to $L_{V_1}$ satisfies the first Beurling--Deny condition;
  \item [(ii)]  For all $\alpha>-\lambda_0(L_{V_1})$, the operator $(L_{V_1}+\alpha)^{-1}$ is positivity preserving in $\lw$.
\end{enumerate}
\end{cor}

\begin{proof} As $(X, b, \mu)$ satisfies the property (FC), we can use Green's formula and remark~\ref{R:R-green} to infer that the quadratic form $u\mapsto (\mathcal{L}_{V_1}u,u)$, with domain $\xcomp$, coincides with the form $Q^{c}_{V_1}$. Moreover (by Green's formula), we see that $\mathcal{L}_{V_1}|_{\xcomp}$ is a symmetric operator in $\lw$. Hence, by an abstract fact, the form $Q^{c}_{V_1}$ is closable, and its closure $Q^{0}_{V_1}$ coincides with the form $h_{1}$. With this in mind, the properties (i) and (ii) of the corollary are restatements of properties (ii) and (iii) of proposition~\ref{P:P-1} respectively.
\end{proof}

The following lemma is concerned with the positive form core property:

\begin{lemma}\label{L-pfc-plus} Let $(X, b, \mu)$ be a weighted graph satisfying the property (FC). Assume that $W\in C(X)$ satisfies $W(x)\geq 0$ for all $x\in X$. Let $H^{0}_{W}$ be as in section~\ref{SS:q-forms}. Then, $H^{0}_{W}$ has a positive form core.
\end{lemma}
\begin{proof} The operator $H^{0}_{W}$, with $W\geq 0$, is the Friedrichs extension of $\mathcal{L}_{W}|_{\xcomp}$. As the corresponding form $Q^{0}_{W}$ is obtained as the closure of $Q^{c}_{W}$, we have the following property: for all $u\in [\dom(Q^{0}_{W})]^{+}$ there exists a sequence $u_k\in \xcomp$ such that $u_k\to u$ in the form norm $\|\cdot\|_{Q^{0}_{W}}$. Denoting by $|\cdot|$ the modulus of a complex number, we see that $v_k:=|u_k|$ is a sequence of elements of $[\xcomp]^{+}$. By proposition 3.4 in~\cite{MS-18}, the form $Q^{0}_{W}$, with $W\geq 0$, is a restriction of the form $Q_{W}$ described at the beginning of section~\ref{SS:q-forms} above. Thus, using the inequality $||\alpha|-|\beta||\leq |\alpha-\beta|$, valid for all $\alpha,\beta\in\mathbb{C}$, we infer that $v_k\to u$ in the form norm $\|\cdot\|_{Q^{0}_{W}}$. This shows that $H^{0}_{W}$ has a positive form core.
\end{proof}

In the sequel we will use the following special case of theorem VIII.3.6 from~\cite{Kato80}:

\begin{prop}\label{P:kato-8} Let $\{t_k\}_{k\in\ZZ+}$ and $t$ be densely defined, closed, lower semi-bounded  forms in a Hilbert space $\mathscr{H}$ satisfying the following following properties:
\begin{enumerate}
  \item [(i)] $\dom(t_k)\subseteq\dom (t)$ for all $k\in\ZZ_{+}$;
  \item [(ii)] $t(u)\leq t_k(u)$  for all $u\in\dom(t_k)$ and all $k\in\ZZ_{+}$;
  \item [(iii)] there is a core $\mathscr{D}$ of $t$ such that $\mathscr{D}\subseteq \dom(t_k)$ for sufficiently large $k$;
  \item [(iv)] $\displaystyle\lim_{k\to\infty}t_k(u)=t(u)$, for all $u\in\mathscr{D}$.
\end{enumerate}
Let  $T_k$ and $T$ be the (lower semi-bounded) self-adjoint operators associated to $t_k$ and $t$ respectively. Then, for all $\alpha>-\lambda_0(t)$ and all $v\in\mathscr{H}$ we have
\begin{equation}\label{E:res-c}
(T_k+\alpha)^{-1}v\to (T+\alpha)^{-1}v,\quad \textrm{in }\mathscr{H}.
\end{equation}
\end{prop}
\begin{remark}\label{R:sr} The statement~(\ref{E:res-c}) is described in the literature as \emph{$T_k$ converges to $T$ in the strong resolvent sense}.
\end{remark}

We are now ready to tackle the positive form core property for the operator $L_{V_1}$ in theorem~\ref{T:main-1}.

\begin{prop}\label{P:P-2} Let $(X, b, \mu)$ be a weighted graph satisfying the property (FC). Assume that $V_1\in C(X)$ is a real-valued function satisfying the condition~(\ref{E:l-v-1}). Let $L_{V_1}$ be the Friedrichs extension of $\mathcal{L}_{V_1}|_{\xcomp}$. Then, $L_{V_1}$ has a positive form core.
\end{prop}
\begin{proof} We begin by writing down the definition of truncated potentials used in the sequel. Define $V^{+}_1:=\max\{V_1,0\}$, $V^{-}_1:=\max\{-V_1,0\}$, and note that $V_1=V^{+}_1-V^{-}_1$. For $k\in\ZZ_+$ define $V_1^{(k)}:=V^{+}_1-\min\{V^{-}_1,k\}$. Let $L_{V^{+}_1}$, $L_{V_1^{(k)}}$, and $L_{V_1}$ be the Friedrichs extensions of $\mathcal{L}_{V^{+}_1}|_{\xcomp}$, $\mathcal{L}_{V_1^{(k)}}|_{\xcomp}$, and  $\mathcal{L}_{V_1}|_{\xcomp}$ respectively. Following the notations of section~\ref{SS:q-forms}, denote the corresponding (closed, lower semi-bounded) quadratic forms by $Q^0_{V^{+}_1}$, $Q^0_{V_1^{(k)}}$, and $Q^0_{V_1}$.

\noindent\emph{Step 1.} Using proposition~\ref{P:kato-8} with $t= Q^0_{V_1}$ and $t_k=Q^0_{V^{(k)}_1}$, we will show that $L_{V_1^{(k)}}\to L_{V_1}$ in the strong resolvent sense (as described in remark~\ref{R:sr}).

In view of the condition~(\ref{E:l-v-1}), we see that $\dom(Q^0_{V^{+}_1})\subseteq \dom(Q^0_{V_1})$. Moreover, since $V^{+}_1-V_1^{(k)}$ belongs to $\lwi$, we see that $\dom(Q^0_{V^{+}_1})=\dom(Q^0_{V^{(k)}_1})$. Thus hypothesis (i) of proposition~\ref{P:kato-8} is satisfied. Furthermore, hypothesis (ii) of the same result is satisfied because $V_1^{(k)}\geq V_1$. By the definitions of the forms $Q^0_{V_1}$ and $Q^0_{V^{(k)}_1}$, hypotheses (iii) and (iv) of proposition~\ref{P:kato-8} are satisfied with $\mathscr{D}=\xcomp$. Thus, we obtain $L_{V_1^{(k)}}\to L_{V_1}$ in the strong resolvent sense.

\medskip

\noindent\emph{Step 2.} As $L_{V^{+}_1}$ has a positive form core (see lemma~\ref{L-pfc-plus}), for all $u\in [\dom(Q^0_{V^{+}_1})]^{+}$, there exists a sequence $\varphi_j\in [\xcomp]^{+}$ such that $\varphi_j \to u$ in the form norm of $Q^0_{V^{+}_1}$. Therefore, due to~(\ref{E:l-v-1}), $\varphi_j \to u$ in the form norm of $Q^0_{V_1}$.

\medskip

\noindent\emph{Step 3.} Based on step 2, the proof of the proposition will be complete after showing the following property: for all $u\in [\dom(Q^0_{V_1})]^{+}$ there exists a sequence $w_j\in [\dom(Q^0_{V^{+}_1})]^{+}$ such that $w_j\to u$ in the form norm of $\dom(Q^0_{V_1})$.
To make the exposition more readable, we assume that the form $Q^0_{V_1}$ is non-negative. With this assumption, the forms  $Q^0_{V^{(k)}_1}$ are also non-negative. In particular, we have $1>-\lambda_0(Q^0_{V_1})=-\lambda_0(L_{V_1})$, and, therefore, for $u\in [\dom(Q^0_{V_1})]^{+}$ and $k,r\in \ZZ_{+}$, the following definitions make sense:
\begin{equation*}
u_r:=\left(r^{-1}L_{V_1}+1\right)^{-1}u,
\end{equation*}
and
\begin{equation*}
u^{k}_r:=\left(r^{-1}L_{V_1^{(k)}}+1\right)^{-1}u.
\end{equation*}

Observe that $u^{k}_r\in \dom(L_{V_1^{(k)}})\subseteq \dom(Q^{0}_{V_1^{(k)}})=\dom(Q^{0}_{V_1^{+}})$. Moreover, in view of corollary~\ref{C:C-1}, we have $u^{k}_r\geq 0$, that is, $u^{k}_r\in[\dom(Q^{0}_{V_1^{+}})]^{+}$. Furthermore,
\begin{align}\label{e-1-2}
&\|u_r-u\|_{Q^0_{V_1}}=\|(L_{V_1}+1)^{1/2}(u_r-u)\|\nonumber\\
&=\left\|\left[\left(r^{-1}L_{V_1}+1\right)^{-1}-1\right](L_{V_1}+1)^{1/2}u\right\|\to 0,
\end{align}
as $r\to\infty$, where the second equality is true because $(L_{V_1}+1)^{1/2}$ and $\left(r^{-1}L_{V_1}+1\right)^{-1}$ commute, and the convergence relation holds because the sequence in $[\cdot]$ converges (in the strong operator sense) to $0$. Therefore, as $r\to\infty$, we have $u_r \to u$ in the form norm of $Q^0_{V_1}$.

Let $Q^0_{V_1+r}$ and $Q^0_{V^{(k)}_1+r}$ be the usual forms $Q^0_{W}$ corresponding to $W=V_1+r$ and $V=V^{(k)}_1+r$ with the form norms
\begin{equation}\label{f-norm-a}
\|w\|_{Q^0_{V_1+r}}^2:=Q^0_{V_1+r}(w)+ (1-r)\|w\|^2,
\end{equation}
\begin{equation*}
\|w\|_{Q^0_{V^{(k)}_1}}^2:=Q^0_{V^{(k)}_1+r}(w)+ (1-r)\|w\|^2.
\end{equation*}
We will next show that for a fixed $r\in\ZZ_{+}$, we have $u^{k}_r\to u_r$, as $k\to \infty$, in the form norm of $Q^0_{V_1}$, that is, the norm in~(\ref{f-norm-a}) corresponding to $r=0$.
We expand
\begin{align}\label{E:left-off-est}
&\left\|\left(L_{V_1^{(k)}}+r\right)^{-1}u-\left(L_{V_1}+r\right)^{-1}u\right\|_{Q^0_{V_1+r}}^2\nonumber\\
&=\left\|\left(L_{V_1^{(k)}}+r\right)^{-1}u\right\|_{Q^0_{V_1+r}}^2 +\left\|\left(L_{V_1}+r\right)^{-1}u\right\|_{Q^0_{V_1+r}}^2\nonumber\\
&-2\left(\left(L_{V_1^{(k)}}+r\right)^{-1}u,\left(L_{V_1}+r\right)^{-1}u\right)_{Q^0_{V_1+r}}
\end{align}
and note that
\begin{equation*}
\left\|\left(L_{V_1^{(k)}}+ r\right)^{-1}u\right\|_{Q^0_{V_1+r}}^2\leq
\left\|\left(L_{V_1^{(k)}}+ r\right)^{-1}u\right\|_{Q^0_{V_1^{(k)}+r}}^2.
\end{equation*}
Keeping in mind~(\ref{f-norm-a}) and the last inequality, we can further estimate and expand~(\ref{E:left-off-est}):
\begin{align}\label{E:left-off-est-2}
&\dots\leq \left\|\left(L_{V_1^{(k)}}+r\right)^{-1}u\right\|_{Q^0_{V^{(k)}_1+r}}^2
+\left\|\left(L_{V_1}+ r\right)^{-1}u\right\|_{Q^0_{V_1+r}}^2\nonumber\\
&-2\left(\left(L_{V_1^{(k)}}+r\right)^{-1}u,\left(L_{V_1}+ r\right)^{-1}u\right)_{Q^0_{V_1+r}}\nonumber\\
&=(1-r)\left\|\left(L_{V_1^{(k)}}+r\right)^{-1}u\right\|^2
+(1-r)\left\|\left(L_{V_1}+r\right)^{-1}u\right\|^2\nonumber\\
&-2(1-r)\left(\left(L_{V_1^{(k)}}+r\right)^{-1}u,\left(L_{V_1}+r\right)^{-1}u\right)\nonumber\\
&+\left(\left(L_{V_1}+r\right)^{-1}u,u\right)-\left(\left(L_{V_1^{(k)}}+r\right)^{-1}u,u\right)\nonumber\\
&\leq \left(\left(L_{V_1}+r\right)^{-1}u,u\right)-\left(\left(L_{V_1^{(k)}}+r\right)^{-1}u,u\right),
\end{align}
where the last inequality holds because the sum of the terms with factor $(1-r)$ is non-positive (as seen from Cauchy--Schwarz inequality and the assumption $r\geq 1$).
Taking the limit as $k\to\infty$ in~(\ref{E:left-off-est-2}) and recalling the strong resolvent convergence $L_{V_1^{(k)}}\to L_{V_1}$, we obtain (for a fixed $r\in\ZZ_{+}$)
\begin{equation*}
\|u^{k}_r\to u_r\|_{Q^0_{V_1+r}}=r\left\|\left(L_{V_1^{(k)}}+r\right)^{-1}u-\left(L_{V_1}+r\right)^{-1}u\right\|_{Q^0_{V_1+r}}\to 0.
\end{equation*}
The first norm in~(\ref{f-norm-a}) yields a family of equivalent form norms of $Q^0_{V_1}$. Therefore, for a fixed $r\in\ZZ_{+}$, we have $u^{k}_r\to u_r$ in the form norm of $Q^0_{V_1}$. Remembering the observation~(\ref{e-1-2}), we infer that there exists a subsequence $\{w_j\}$ of $\{u^{k}_{r}\}$ such that $w_j\to u$ in the form norm of $Q^0_{V_1}$.
\end{proof}

\section{Proof of Theorem~\ref{T:main-1}}
We begin this section with the following lemma for which we refer the reader to corollary 2.3 in~\cite{MS-18}:
\begin{lem}\label{L:L-2} Let $(X, b, \mu)$ be a weighted graph (not necessarily satisfying the property (FC)). Let $W\in C(X)$ and $W_0\in C(X)$ be functions such that $W(x)\geq W_0(x)$ for all $x\in X$. Let  $\mathcal{F}$ and $\mathcal{L}_{W}$ be as in~(\ref{E:def-f}) and~(\ref{E:magnetic-lap}) respectively. Assume that $f\in \mathcal{F}$ and $\beta\in\mathbb{R}$ satisfy $\mathcal{L}_{W}f=\beta f$. Then, we have the following (pointwise) inequality: $\mathcal{L}_{W_0}|f|\leq \beta |f|$.
\end{lem}

Assuming that $(X, b, \mu)$ satisfies the property (FC), we proceed with the proof of theorem~\ref{T:main-1}. Recall that the form associated to the operator $L_{V_1}\widetilde{+}M_{V_2}$ is the sum $Q^0_{V_1}+q_{V_2}$. Additionally, note that $(Q^0_{V_1}+q_{V_2})|_{\xcomp}=Q^{c}_{V}$, with $V=V_1+V_2$. Therefore, the (closed) form $Q^0_{V_1}+q_{V_2}$ is an extension of the form $Q^0_{V}$, where $Q^0_{V}$ is as in the text above section~\ref{SS:pfc}. Thus, in order to show the equality $L_{V_1}\widetilde{+}M_{V_2}=L_{V}$, it is enough to show that $\xcomp$ is dense in $\dom (Q^0_{V_1}+q_{V_2})=\dom(Q^0_{V_1})\cap \dom(q_{V_2})$ with respect to the norm
\begin{equation*}
\|w\|_{V}^2:=Q^{0}_{V_1}(w)+q_{V_2}(w)+\alpha\|w\|^2,
\end{equation*}
where $\alpha>-\lambda_0(Q^{0}_{V_1})$.

We will accomplish this by showing that if $u\in \dom(Q^0_{V_1})\cap \dom(q_{V_2})$ satisfies
\begin{equation}\label{hyp-u}
Q^{0}_{V_1}(u,v)+q_{V_2}(u,v)+\alpha(u,v)=0,
\end{equation}
for all $v\in\xcomp$, then $u=0$.

By lemma~\ref{L:L-1} and remark~\ref{R:R-green}, we can rewrite~(\ref{hyp-u}) as
\begin{equation*}
0=(\mathcal{L}_{V_1}u,v)_{a}+(V_2u,v)_{a}+\alpha(u,v)=(\mathcal{L}_{V}u,v)_{a}+\alpha(u,v),
\end{equation*}
where $(\cdot,\cdot)_{a}$ is as in~(\ref{E:a-d}).

Substituting $v=\delta_{x}$, where $\delta_{x}(y)=1$ if $x=y$ and $\delta_{x}(y)=0$ otherwise, we get
$\mathcal{L}_{V}u=-\alpha u$. Keeping in mind that $V_2\geq 0$ and $V_1=V_1+V_2$, we see that $V\geq V_1$. Hence, by lemma~\ref{L:L-2}, the equation $\mathcal{L}_{V}u=-\alpha u$ leads to $\mathcal{L}_{V_1}|u|\leq -\alpha|u|$, that is,
\begin{equation}\label{E:temp-1}
((\mathcal{L}_{V_1}+\alpha)|u|, v)_{a}\leq 0,
\end{equation}
for all $v\in [\xcomp]^{+}$.

As  $u\in \dom(Q^0_{V_1})$ and since the form $Q^c_{V_1}$ is lower semi-bounded,
we can use corollary~\ref{C:C-1} to infer that $|u|\in \dom(Q^0_{V_1})$. With this information at hand, lemma~\ref{L:L-1} enables us to rewrite~(\ref{E:temp-1}) as
\begin{equation}\label{temp-2}
(|u|, v)_{Q^0_{V_1}}\leq 0,
\end{equation}
for all $v\in [\xcomp]^{+}$, where $(\cdot,\cdot)_{Q^0_{V_1}}$ is the inner product corresponding to the form norm of $Q^0_{V_1}$, as described in~(\ref{E:n-alpha}).

Defining $w:=\left(L_{V_1}+\alpha\right)^{-1}|u|$, we see that $w\in [\dom(L_{V_1})]^{+}\subseteq [\dom(Q^0_{V_1})]^{+}$ because $\left(L_{V_1}+\alpha\right)^{-1}$ is positivity preserving in view of corollary~\ref{C:C-1}. Furthermore, by proposition~\ref{P:P-2}, $L_{V_1}$ has a positive form core, which means that there exists $w_k\in [\xcomp]^{+}$ such that $w_k\to w$, as $k\to\infty$, in the form norm of $Q^0_{V_1}$. Therefore, keeping in mind~(\ref{temp-2}), we have
\begin{equation*}
0\geq \displaystyle \lim_{k\to\infty}(|u|, w_k)_{Q^0_{V_1}}=(|u|, w)_{Q^0_{V_1}}=(|u|, \left(L_{V_1}+\alpha\right)^{-1}|u|)_{Q^0_{V_1}}=(|u|,|u|),
\end{equation*}
which shows that $u=0$. $\hfill\square$

\section{Proof of theorem~\ref{T:main-2}} Throughout this section, $\overline{G}$ is the closure of a (closable) operator $G$ in $\lw$. To simplify the discussion, without losing generality we assume that~(\ref{E:l-v-1}) is satisfied with $C=0$.
Let $L_{V_1}$ and $L_{V}$ be the Friedrichs extensions of $\mathcal{L}_{V_1}|_{\xcomp}$ and $\mathcal{L}_{V}|_{\xcomp}$ respectively, and let $M_{V_2}$ be the maximal multiplication operator in $\lw$ corresponding to $V_2$, with the domain as in~(\ref{d-v-2}).


We begin with the following lemma:
\begin{lem}\label{L:2-3} Assume that the hypotheses of theorem~\ref{T:main-2} are satisfied. Then, $\dom(L_{V_1})\subseteq\dom(\overline{\mathcal{L}_{V}|_{\xcomp}})$.
Moreover, the following equality holds for all $u\in \dom (L_{V_1})$:
\begin{equation}\label{E:l-v-1-2-new}
L_{V}u=L_{V_1}u+M_{V_2}u,
\end{equation}
where $M_{V_2}$ is as in~(\ref{d-v-2}).
\end{lem}
\begin{proof} By the hypothesis on $\mathcal{L}_{V_1}|_{\xcomp}$, we have
\begin{equation}\label{temp-closure}
\overline{\mathcal{L}_{V_1}|_{\xcomp}}=L_{V_1}
\end{equation}
With this in mind, the property $\dom (L_{V_1})\subseteq\dom\left(\overline{\mathcal{L}_{V}|_{\xcomp}}\right)$ follows from the hypothesis~(\ref{E:a-1-2}). Having established this inclusion, the property~(\ref{E:l-v-1-2-new}) follows easily if we take into account~(\ref{temp-closure}) and the following two facts: the operator $M_{V_2}$ is closed and $L_{V}$ is an extension of  $\overline{\mathcal{L}_{V}|_{\xcomp}}$, where the latter property is true because $L_{V}$ the Friedrichs extension of $\mathcal{L}_{V}|_{\xcomp}$.
\end{proof}

The proof of theorem~\ref{T:main-2} rests on the following proposition, whose proof occupies the remainder of this section.
\begin{prop}\label{P:2-1} Assume that the hypotheses of theorem~\ref{T:main-2} are satisfied. Then,
\begin{enumerate}
  \item [(i)] $\dom(L_{V_1})\subseteq \dom(L_{V})$;
  \item [(ii)] for all  $u\in\dom(L_{V})$, there exists a sequence $w_j\in \dom(L_{V_1})$ such that $w_j\to u$ and $L_{V}w_j\to L_{V}u$, where both limits are understood with respect to the norm of $\lw$.
\end{enumerate}
\end{prop}
\begin{rem} As $L_{V}$ is an extension of $\overline{\mathcal{L}_{V}|_{\xcomp}}$, property (i) follows right away from lemma~\ref{L:2-3}.
\end{rem}

To prove the property (ii) of this proposition, we will use truncated operators. For $k\in\ZZ_+$, define $V_{2}^{(k)}:=\min\{V_2,k\}$ and $T_{k}:=L_{V_1}+M_{V_{2}^{(k)}}$, where $M_{V_{2}^{(k)}}$ is as in~(\ref{d-v-2}) with $V_2$ replaced by $V_{2}^{(k)}$, and the sum is understood in the operator sense. Note that $\dom(T_k)=\dom(L_{V_1})\cap \dom(M_{V_{2}^{(k)}})=\dom(L_{V_1})$, where the last equality is true because $V_{2}^{(k)}\in\lwi$.

As $M_{V_{2}^{(k)}}$ is a bounded operator in~$\lw$, the following lemma is a simple consequence of theorem V.4.4 in~\cite{Kato80}, known as Kato--Rellich theorem:

\begin{lem}\label{L:2-1} Assume that the hypotheses of theorem~\ref{T:main-2} are satisfied. Then, for all $k\in\ZZ+$ the operator $T_{k}$ is self-adjoint, and $T_k=\overline{\mathcal{L}_{V_1+V_{2}^{(k)}}|_{\xcomp}}$.
\end{lem}

For $u\in\dom(L_{V})$ and $k,r\in \ZZ_{+}$, define
\begin{equation}\label{u-k-r-t-k}
u^{k}_r:=\left(r^{-1}T_k+1\right)^{-1}u.
\end{equation}
In view of the hypothesis~(\ref{E:l-v-1}) with $C=0$ and since $V_{2}^{(k)}\geq 0$, we have $1>-\lambda_0(T_k)$; hence, this definition makes sense. Observe that $u^{k}_r\in\dom(T_k)=\dom (L_{V_1})$.

\begin{lem}\label{L:2-2} Assume that the hypotheses of theorem~\ref{T:main-2} are satisfied. Let $u\in\dom(L_{V})$ and let $u^{k}_r$ be as in~(\ref{u-k-r-t-k}). Then, $|u^{k}_r|\leq \left(r^{-1}L_{V_1}+1\right)^{-1}|u|$.
\end{lem}
\begin{proof} In view of the hypothesis~(\ref{E:l-v-1}) with $C=0$ we have $1>-\lambda_0(L_{V_1})$; hence the expression $\left(r^{-1}L_{V_1}+1\right)^{-1}|u|$ makes sense. Using the inequality $V_1+V_{2}^{(k)}\geq V_1$, it follows (see lemma 3.12 in~\cite{MS-18}) that the resolvent $\left(r^{-1}L_{V_1}+1\right)^{-1}$ dominates the resolvent $\left(r^{-1}T_k+1\right)^{-1}$:
\begin{equation*}
|u^{k}_r|:=\left|\left(r^{-1}T_k+1\right)^{-1}u\right|\leq \left(r^{-1}L_{V_1}+1\right)^{-1}|u|.
\end{equation*}
\end{proof}

\begin{lem}\label{L:2-4} Assume that the hypotheses of theorem~\ref{T:main-2} are satisfied. Let $u\in\dom(L_{V})$ and let $u^{k}_r$, $k,r\in \ZZ_{+}$, be as in~(\ref{u-k-r-t-k}).  Then, for fixed $r\in\ZZ_{+}$, we have the following convergence relations in $\lw$, as $k\to\infty$:
\begin{equation*}
u^{k}_r\to \left(r^{-1}L_{V}+1\right)^{-1}u,\qquad L_{V}u^{k}_r \to L_{V}\left(r^{-1}L_{V}+1\right)^{-1}u.
\end{equation*}
\end{lem}
\begin{proof} Throughout this proof, $r\in\ZZ_{+}$ is fixed. Remembering that $u^{k}_r\in\dom(T_k)=\dom(L_{V_1}))$ and taking into account lemma~\ref{L:2-3}, we see that
\begin{equation}\label{temp-c-1}
L_{V}u^{k}_r=L_{V_1}u^{k}_r+M_{V_2}u^{k}_r,
\end{equation}
which leads to
\begin{equation}\label{temp-c-2}
u^{k}_r-\left(r^{-1}L_{V}+1\right)^{-1}u=r^{-1}\left(r^{-1}L_{V}+1\right)^{-1}(V_2-V_{2}^{(k)})u^{k}_r.
\end{equation}

The operator $\left(r^{-1}L_{V}+1\right)^{-1}$ is bounded (and linear) in $\lw$. By the definition of $V_{2}^{(k)}$ we have the pointwise convergence $(V_2-V_{2}^{(k)})u^{k}_r\to 0$, as $k\to\infty$. Thus, the dominated convergence theorem will grant the convergence to $0$ of the right hand side of~(\ref{temp-c-2}) in $\lw$  after we show that the sequence $\{(V_2-V_{2}^{(k)})u^{k}_r\}_{k\in\ZZ_+}$ is majorized by a function in $\lw$. To this end, using lemma~\ref{L:2-2} we estimate
\begin{equation*}
|(V_2-V_{2}^{(k)})u^{k}_r|\leq (V_2-V_{2}^{(k)})\left(r^{-1}L_{V_1}+1\right)^{-1}|u|\leq V_2\left(r^{-1}L_{V_1}+1\right)^{-1}|u|,
\end{equation*}
where the last inequality follows because $V_2-V_{2}^{(k)}\leq V_2$.

Using the the hypothesis~(\ref{E:a-1-2}) and property~(\ref{temp-closure}), it is easily seen that function $V_2\left(r^{-1}L_{V_1}+1\right)^{-1}|u|$ belongs to $\lw$. This concludes the proof of the first convergence relation.

To establish the second convergence relation, we look at~(\ref{temp-c-1}) and write
\begin{equation*}
L_{V}u^{k}_r=L_{V_1}u^{k}_r+M_{V_{2}^{(k)}}u^{k}_r+(V_2-V_{2}^{(k)})u^{k}_r=T_ku^{k}_r+(V_2-V_{2}^{(k)})u^{k}_r,
\end{equation*}
where $T_k$ is as in lemma~\ref{L:2-1}. Using the convergence $(V_2-V_{2}^{(k)})u^{k}_r\to 0$ in $\lw$ as $k\to\infty$, it remains to show
\begin{equation}\label{c-rel-final}
T_ku^{k}_r\to L_{V}\left(r^{-1}L_{V}+1\right)^{-1}u,\quad \textrm{in }\lw.
\end{equation}
The convergence relation~(\ref{c-rel-final}) follows by combining
\begin{equation*}
L_{V}\left(r^{-1}L_{V}+1\right)^{-1}u=ru-r\left(r^{-1}L_{V}+1\right)^{-1}u
\end{equation*}
and
\begin{equation*}
T_ku^{k}_r=ru-ru^{k}_r\to ru-r\left(r^{-1}L_{V}+1\right)^{-1}u,
\end{equation*}
where we used the first convergence relation from the statement of this lemma.
\end{proof}
\begin{lem}\label{L:2-5} Assume that the hypotheses of theorem~\ref{T:main-2} are satisfied. Let $u\in\dom(L_{V})$ and let $r\in \ZZ_{+}$. Then, we have the following convergence relations in $\lw$, as $r\to\infty$:
\begin{equation*}
\left(r^{-1}L_{V}+1\right)^{-1}u\to u,\qquad L_{V}\left(r^{-1}L_{V}+1\right)^{-1}u\to L_{V}u.
\end{equation*}
\end{lem}
\begin{proof} Let $r^{-1}Q_{V}$ be the (closed, non-negative) form associated to the (self-adjoint non-negative) operator $r^{-1}L_{V}$. Then the sequence of forms $t_{r}:=r^{-1}Q_{V}$, $r\in\ZZ_+$, and the zero-form $t\equiv 0$ satisfy the hypotheses of proposition~\ref{P:kato-8} with $\mathscr{D}=\xcomp$, which establishes the first convergence relation. As $u\in\dom(L_V)$, we can use the (abstract) result of the problem III.6.2 in~\cite{Kato80} to infer
\begin{equation*}
L_{V}\left(r^{-1}L_{V}+1\right)^{-1}u=\left(r^{-1}L_{V}+1\right)^{-1}L_{V}u.
\end{equation*}
Now the second convergence relation follows from proposition~\ref{P:kato-8}.
\end{proof}

\noindent\textbf{Proof of Proposition~\ref{P:2-1}} It remains to show property (ii). Let $u\in\dom(L_{V})$. By lemma~\ref{L:2-4}, for every $r\in\ZZ_{+}$, we can select $k(r)\in\ZZ_{+}$ such that
\begin{equation*}
\|u^{k(r)}_r-\left(r^{-1}L_{V}+1\right)^{-1}u\|\leq r^{-1},\quad \|L_{V}u^{k(r)}_r \to L_{V}\left(r^{-1}L_{V}+1\right)^{-1}u\|\leq r^{-1}.
\end{equation*}
Defining $w_{r}:=u^{k(r)}_r$, we see that $w_r\in\dom (T_{k(r)})=\dom (L_{V_1})$. Combining the last two estimates with lemma~\ref{L:2-5}, we get $w_r\to u$ and $L_{V}w_r\to L_{V}u$ in $\lw$, as $r\to\infty$. $\hfill\square$

\medskip

\noindent\textbf{End of the Proof of Theorem~\ref{T:main-2}} Since $L_{V}$ is an extension of $\overline{\mathcal{L}_{V}|_{\xcomp}}$, it is enough to show that $\dom (L_{V})\subseteq \dom\left(\overline{\mathcal{L}_{V}|_{\xcomp}}\right)$. Let $u\in \dom (L_V)$ and let $w_j\in\dom (L_{V_1})$ be a sequence as in part (ii) of proposition~\ref{P:2-1}. By lemma~\ref{L:2-3} we have $w_j\in\dom (\overline{\mathcal{L}_{V}|_{\xcomp}})$. As $\overline{\mathcal{L}_{V}|_{\xcomp}}$ is a closed operator, we infer that $u\in \dom\left(\overline{\mathcal{L}_{V}|_{\xcomp}}\right)$. $\hfill\square$

\vskip 0.25in

\noindent\textbf{Data Availability} Data sharing is not applicable to this article as no data sets
were generated or analyzed during this study.

\vskip 0.25in

\noindent\textbf{Declarations}

\vskip 0.25in

\noindent\textbf{Conflict of Interest} The author has no conflict of interest to declare that
are relevant to the content of this article.

\end{document}